\documentclass{amsart}

\usepackage[english]{babel}
\usepackage[utf8]{inputenc}

\usepackage{amsmath,amsfonts,amssymb,amscd,amsthm}
\usepackage[linktocpage,unicode]{hyperref}

\usepackage{latexsym,graphics}
\usepackage[dvips]{color}
\usepackage[dvips]{graphicx}
\usepackage{pstricks,pst-tree,rotating}

\title[Cumulants of the $q$-semicircular law,  Tutte polynomials, and heaps]
  {Cumulants of the $q$-semicircular law, \\ Tutte polynomials, and heaps}
\date{}
\author{Matthieu Josuat-Vergès}
\thanks{This work was supported by the Austrian Research Foundation (FWF) via the grant Y463}
\address{CNRS and Institut Gaspard Monge, Université Paris-Est Marne-la-Vallée\\
5 Boulevard Descartes\\
Champs-sur-Marne\\
77454 Marne-la-Vallée cedex 2\\ France}
\email{matthieu.josuat-verges@univ-mlv.fr}

\hypersetup{
    unicode=true,          
    pdftoolbar=true,        
    pdfmenubar=true,        
    pdffitwindow=false,     
    pdfstartview={FitH},    
    pdftitle={Cumulants of the q-semicircular law, Tutte polynomials, and heaps},
    pdfauthor={Matthieu Josuat-Vergès},
    pdfsubject={},
    pdfkeywords={moments, cumulants, matchings, Tutte polynomials, heaps},
    pdfnewwindow=true,      
    colorlinks=true,       
    linkcolor=red,          
    citecolor=blue,        
    filecolor=magenta,      
    urlcolor=cyan           
}

\newtheorem{thm}{Theorem}[section]

\newtheorem{lem}[thm]{Lemma}
\newtheorem{prop}[thm]{Proposition}

\theoremstyle{definition}
\newtheorem{defn}[thm]{Definition}
\newtheorem{rem}[thm]{Remark}

\DeclareMathOperator{\cro}{cr}

\begin{document}

\begin{abstract}
The $q$-semicircular distribution is a probability law that interpolates between the Gaussian law 
and the semicircular law. There is a combinatorial interpretation of its moments in terms of matchings 
where $q$ follows the number of crossings, whereas for the free cumulants one has to restrict the 
enumeration to connected matchings.
The purpose of this article is to describe combinatorial properties of the classical cumulants.
We show that like the free cumulants, they are obtained by an enumeration of connected matchings,
the weight being now an evaluation of the Tutte polynomial of a so-called crossing graph.
The case $q=0$ of these cumulants was studied by Lassalle using symmetric functions and
hypergeometric series. We show that the underlying combinatorics is explained through the
theory of heaps, which is Viennot's geometric interpretation of the Cartier-Foata monoid.
This method also gives results for the classical cumulants of the free Poisson law.
\end{abstract}

\maketitle

\tableofcontents

\section{Introduction}

Let us consider the sequence $\{m_n(q)\}_{n\geq0}$ defined by the generating function
\[ 
  \sum_{n\geq 0}  m_n(q) z^n =
   \cfrac{1}{1 -
   \cfrac{ [1]_q z^2}{1 -
   \cfrac{ [2]_q z^2}{1 - \ddots
}}}
\]
where $[i]_q=\frac{1-q^i}{1-q}$. For example, $m_0(q)=m_2(q)=1$, $m_4(q)=2+q$, and the odd values are 0.
The generating function being a Stieltjes continued fraction, $m_n(q)$ is the $n$th moment of a symmetric
probability measure on $\mathbb{R}$ (at least when $0\leq q\leq 1$). An explicit formula for the density 
$w(x)$ such that $m_n(q)=\int x^n w(x) {\rm d}x$ is given by Szegő~\cite{szego}:
\[
  w(x) =
 \begin{cases} \frac 1\pi \sqrt{1-q} \sin\theta \prod\limits_{n=1}^\infty (1-q^n)  |1-q^ne^{2i\theta}|^2  & 
  \text{ if } -2\leq x\sqrt{1-q} \leq 2, \\  0 & \text{otherwise,}
 \end{cases}
\]
where $\theta\in[0,\pi]$ is such that $2 \cos \theta = x \sqrt{1-q}$.
At $q=0$, it is the semicircular distribution with density $(2\pi)^{-1}\sqrt{4-x^2}$ supported on $[-2,2]$, 
whereas at the limit $q\to 1$ it becomes the Gaussian distribution with density $(2\pi)^{-1/2}e^{-x^2/2}$.
This law is therefore known either as the $q$-Gaussian or the $q$-semicircular law.
It can be conveniently characterized by its orthogonal polynomials, defined by the relation
$xH_n(x|q) = H_{n+1}(x|q) + [n]_q H_{n-1}(x|q)$ together with $H_1(x|q)=x$ and $H_0(x|q)=1$,
and called the continuous $q$-Hermite polynomials (but we do not insist on this point of view
since the notion of cumulant is not particularly relevant for orthogonal polynomials).

The semicircular law is the analogue in free probability of the Gaussian law \cite{hiai,nica}. 
More generally, the $q$-semicircular measure plays an important role in noncommutative probability 
theories \cite{anshelevitch,blitvic,bozejko1,bozejko2,leeuwen1,leeuwen2}. This was initiated by Bożejko
and Speicher \cite{bozejko1,bozejko2} who used creation and annihilation operators in a twisted 
Fock space to build generalized Brownian motions. 

The goal of this article is to examine the combinatorial meaning of the classical
cumulants $k_n(q)$ of the $q$-semicircular law (we recall the definition in the next section). 
The first values lead to the observation that
\begin{equation*}
 \tilde k_{2n}(q) =  \frac{ k_{2n}(q) }{ (q-1)^{n-1} }
\end{equation*}
is a polynomial in $q$ with nonnegative coefficients. For example:
\begin{equation*}
  \tilde k_2(q)=\tilde k_4(q)=1, \qquad  \tilde k_6(q)=q+5, \qquad
  \tilde k_8(q)= q^3+7q^2+28q+56.
\end{equation*}
We actually show in Theorem~\ref{cumultutte} that this $ \tilde k_{2n}(q)$ can be given a meaning
as a generating function of connected matchings, i.e. the same objects that give a 
combinatorial meaning to the free cumulants of the $q$-semicircular law. However, the weight function 
that we use here on connected matching is not as simple as in the case of free cumulants, it is 
given by the value at $(1,q)$ of the Tutte polynomial of a graph attached to each connected matching,
called the crossing graph.

There are various points where the evaluation of a Tutte polynomials has combinatorial meaning,
in particular $(1,0)$, $(1,1)$ and $(1,2)$. In the first and third case ($q=0$ and $q=2$), they 
can be used to give an alternative proof of Theorem~\ref{cumultutte}. These will be provided
respectively in Section~\ref{sec:heaps} and Section~\ref{sec:q=2}. The integers $\tilde k_{2n}(0)$
were recently considered by Lassalle \cite{lassalle} who defines them as a sequence simply related with
Catalan numbers, and further studied in \cite{vignat}.
Being the (classical) cumulants of the semicircular law, it might seem unnatural to 
consider this quantity since this law belongs to the world of free probability, but on the other side
the free cumulants of the Gaussian have numerous properties (see \cite{belinschi}). The interesting 
feature is that this particular case $q=0$ can be proved via the theory of heaps \cite{cartier,viennot}.
As for the case $q=2$, even though the $q$-semicircular is only defined when $|q|<1$ its moments and
cumulants and the link between still exist because \eqref{relmk} can be seen as an identity between formal
power series in $z$. The particular proof for $q=2$ is an application of the exponential formula.

\section{Preliminaries}

Let us first precise some terms used in the introduction. Besides the moments $\{m_n(q)\}_{n\geq0}$, the
$q$-semicircular law can be characterized by its {\it cumulants} $\{k_n(q)\}_{n\geq1}$ formally defined by
\begin{equation} \label{relmk}
  \sum_{n\geq 1}  k_n(q)  \frac{z^n}{n!} = \log \Bigg(  \sum_{n\geq 0}  m_n(q)  \frac{z^n}{n!}  \Bigg),
\end{equation}
or by its {\it free cumulants} $\{c_n(q)\}_{n\geq1}$ \cite{nica} formally defined by
\[
 C(zM(z)) = M(z)\quad \text{ where } M(z)=\sum_{n\geq0} m_n(q)z^n,\quad C(z) = 1+\sum_{n\geq1} c_n(q) z^n.
\]
These relations can be reformulated using set partitions.

For any finite set $V$, let $\mathcal{P}(V)$ denote the lattice of set partitions of $V$, and let
$\mathcal{P}(n)=\mathcal{P}(\{1,\dots,n\})$. We will denote by $\hat 1$ the
maximal element and by $\mu$ the Möbius function of these lattices, without mentioning $V$ explicitly.
Although we will not use it, let us mention that $\mu(\pi,\hat 1) = (-1)^{\# \pi -1} (\#\pi -1)!$ where $\#\pi$ 
is the number of blocks in $\pi$. See \cite[Chapter~3]{stanley} for details. 
When we have some sequence $(u_n)_{n\geq 1}$, for any $\pi\in\mathcal{P}(V)$ we will use the notation:
\[
   u_\pi = \prod_{ b \in \pi } u_{\# b}.
\]
Then the relations between moments and cumulants read:
\begin{equation} \label{inversion}
  m_n(q) = \sum_{ \pi \in \mathcal{P}(n) } k_\pi(q), \qquad
  k_n(q)  = \sum_{ \pi \in \mathcal{P}(n) } m_\pi(q) \mu(\pi,\hat 1).
\end{equation}
These are equivalent via the Möbius inversion formula and both can be obtained from \eqref{relmk} using
Faà di Bruno's formula. 
When $V\subset\mathbb{N}$, let $\mathcal{NC}(V)\subset\mathcal{P}(V)$ denote the subset of {\it noncrossing partitions}, 
which form a sublattice with Möbius function $\mu^{NC}$. Then we have \cite{hiai,nica}:
\begin{equation} \label{inversionfree}
  m_n(q) = \sum_{ \pi \in \mathcal{NC}(n) } c_\pi(q), \qquad
  c_n(q)  = \sum_{ \pi \in \mathcal{NC}(n) } m_\pi(q) \mu^{NC}(\pi,\hat 1).
\end{equation}
Equations \eqref{inversion} and \eqref{inversionfree} can be used to compute the first non-zero values:
\begin{equation*}\begin{array}{lll}
  k_2(q)=1,\qquad & k_4(q)=q-1, \qquad & k_6(q) =q^3+3q^2-9q+5, \\[2mm]
  c_2(q)=1,\qquad & c_4(q)=q,    \qquad & c_6(q) =q^3+3q^2.
\end{array}\end{equation*}

Let $\mathcal{M}(V)\subset\mathcal{P}(V)$ denote the set of {\it matchings}, i.e. set partitions whose all blocks
have size 2. As is customary, a block of $\sigma\in\mathcal{M}(V)$ will be called an {\it arch}. When
$V\subset\mathbb{N}$, a {\it crossing} \cite{ismail} of $\sigma\in\mathcal{M}(V)$ is a pair of arches $\{i,j\}$ and $\{k,\ell\}$
such that $i<k<j<\ell$. Let $\cro(\sigma)$ denote the number of crossings of $\sigma\in\mathcal{M}(V)$. Let
$\mathcal{N}(V) = \mathcal{M}(V) \cap \mathcal{NC}(V)$ denote the set of {\it noncrossing matchings},
i.e. those such that $\cro(\sigma)=0$.
Let also $\mathcal{M}(2n) = \mathcal{M}(\{1,\dots,2n\})$ and $\mathcal{N}(2n) = \mathcal{N}(\{1,\dots,2n\})$.
Let $\mathcal{P}^c(n) \subset \mathcal{P}(n)$ denote the set of {\it connected} set partitions, i.e. $\pi$ such 
that no proper interval of $\{1,\dots,n\}$ is a union of blocks of $\pi$, and let 
$\mathcal{M}^c(2n) = \mathcal{M}(2n) \cap \mathcal{P}^c(2n)$ denote the set of connected matchings.

It is known \cite{ismail} that for any $n\geq0$, the moment $m_{2n}(q)$ count matchings on $2n$
points according to the number of crossings:
\begin{equation} \label{mucro}
  m_{2n}(q) = \sum_{\sigma\in\mathcal{M}(2n)} q^{\cro(\sigma)}.
\end{equation}
It was showed by Lehner~\cite{lehner} that \eqref{inversionfree} and \eqref{mucro} gives a combinatorial meaning for the 
free cumulants:
\[
  c_{2n}(q) = \sum_{\sigma\in\mathcal{M}^c(2n)} q^{\cro(\sigma)}.
\]
See \cite{belinschi} for various properties of connected matchings in the context of free
probability. Let us also mention that both quantities $m_{2n}(q)$ and $c_{2n}(q)$
are considered in an article by Touchard \cite{touchard}.

\section{\texorpdfstring{A combinatorial formula for $k_n(q)$}{A combinatorial formula for kn(q)} }

We will use the Möbius inversion formula in Equation~\eqref{inversion},
but we first need to consider the combinatorial meaning of the products $m_{\pi}(q)$.

\begin{lem} \label{lemmpi}
For any $\sigma\in\mathcal{M}(2n)$ and $\pi\in\mathcal{P}(2n)$  such that $\sigma \leq \pi$,
let  $\cro(\sigma,\pi)$ be the number of crossings $(\{i,j\},\{k,\ell\})$ of $\sigma$ such that $\{i,j,k,\ell\}\subset b$
for some $b\in\pi$. Then we have:
  \begin{equation} \label{mpi}
     m_\pi(q) = \sum_{ \substack{  \sigma \in\mathcal{M}(2n) \\  \sigma \leq \pi} } q^{\cro(\sigma,\pi)}.
  \end{equation}
\end{lem}

\begin{proof}
 Denoting $\sigma|_b = \{ x\in\sigma \; : \; x\subset b \}$,
 the map $\sigma \mapsto (\sigma|_b)_{b\in\pi}$ is a natural bijection between the set
 $\{\sigma \in\mathcal{M}(2n) \; : \;  \sigma\leq\pi \}$ and the product
 $\Pi_{b\in\pi} \mathcal{M}(b) $, in  such a way that $\cro(\sigma,\pi) = \sum_{b\in\pi} \cro (\sigma|_b)$.
 This allows to factorize the right-hand side in \eqref{mpi} and obtain $m_{\pi}(q)$.
\end{proof}

From Equation~\eqref{inversion} and the previous lemma, we have:
\begin{equation} \label{kW}   \begin{split}
   k_{2n}(q)  &= \sum_{ \pi \in \mathcal{P}(2n) } m_\pi(q) \mu(\pi,\hat 1)
            = \sum_{ \pi \in \mathcal{P}(2n) }  \sum_{\substack{ \sigma \in \mathcal{M}(2n) \\ \sigma \leq \pi}}
               q^{\cro(\sigma,\pi)} \mu(\pi,\hat 1) \\
    &= \sum_{ \sigma \in \mathcal{M}(2n) }  \sum_{ \substack{ \pi\in\mathcal{P}(2n)  \\  \pi \geq \sigma}}
         q^{\cro(\sigma,\pi)} \mu(\pi,\hat 1)  = \sum_{ \sigma \in \mathcal{M}(2n) } W(\sigma),
\end{split}\end{equation}
where for each $\sigma\in\mathcal{M}(2n)$ we have introduced:
\begin{equation}  \label{W1}
   W(\sigma) = \sum_{\substack{ \pi\in\mathcal{P}(2n) \\ \pi \geq \sigma}} q^{\cro(\sigma,\pi)} \mu(\pi,\hat 1).
\end{equation}

A key point is to note that $W(\sigma)$ only depends on how the arches of $\sigma$ cross
with respect to each other, which can be encoded in a graph. This leads to the following:

\begin{defn}
  Let $\sigma\in\mathcal{M}(2n)$. The {\it crossing graph} $G(\sigma)=(V,E)$ is as follows.
  The vertex set $V$ contains the arches of $\sigma$ (i.e. $V=\sigma$), and the edge set $E$
  contains the crossings of $\sigma$ (i.e. there is an edge between the vertices $\{i,j\}$ and $\{k,\ell\}$
  if and only if $i<k<j<\ell$).
\end{defn}

See Figure~\ref{crogra} for an example. Note that the graph $G(\sigma)$ is connected if and only if
$\sigma$ is a connected matching in the sense of the previous section.

\begin{figure}[h!tp] \psset{unit=4mm}
\begin{pspicture}(1,-1)(12,3.5)
  \psdots(1,0)(2,0)(3,0)(4,0)(5,0)(6,0)(7,0)(8,0)(9,0)(10,0)(11,0)(12,0)
  \rput(1,-1){\tiny 1}\rput(2,-1){\tiny 2}\rput(3,-1){\tiny 3}\rput(4,-1){\tiny 4}\rput(5,-1){\tiny 5}
  \rput(6,-1){\tiny 6}\rput(7,-1){\tiny 7}\rput(8,-1){\tiny 8}\rput(9,-1){\tiny 9}\rput(10,-1){\tiny 10}
  \rput(11,-1){\tiny 11}\rput(12,-1){\tiny 12}
  \psarc(8,0){1}{0}{180}\psarc(11,0){1}{0}{180}
  \psarc(4,0){1}{0}{180}\psarc(7.5,0){3.5}{0}{180}
  \psarc(3.5,0){2.5}{0}{180}
  \psarc(5,0){3}{0}{180}
\end{pspicture}
\hspace{1.5cm} \psset{unit=8mm}
\begin{pspicture}(4,3)
  \psdots(1,1)(2,1)(1.5,1.87)(0.29,0.29)(2,0)(2.9,1.4)
  \psline[arrowlength=2]{-}(1,1)(1.5,1.87)\psline[arrowlength=2]{-}(2,1)(1.5,1.87)
  \psline[arrowlength=2]{-}(1,1)(2,1)\psline[arrowlength=2]{-}(1,1)(0.29,0.29)
  \psline[arrowlength=2]{-}(2,1)(2,0)\psline[arrowlength=2]{-}(2,1)(2.9,1.4)
  \rput(1.5,2.1){\tiny \{$1,6$\}}\rput(0.5,1.1){\tiny \{$2,8$\}}
  \rput(0,0){\tiny \{$7,9$\}}\rput(2,-0.3){\tiny \{$3,5$\}}
  \rput(2.55,0.75){\tiny \{$4,11$\}}\rput(3.6,1.5){\tiny \{$10,12$\}}
\end{pspicture}
\caption{A matching $\sigma$ and its crossing graph $G(\sigma)$. \label{crogra}}
\end{figure}
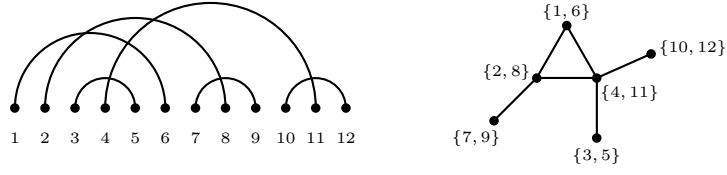

\begin{lem} \label{Wgraph}
 Let $\sigma\in\mathcal{M}(2n)$ and $G(\sigma)=(V,E)$ be its crossing graph.
 If $\pi\in\mathcal{P}(V)$, let $i(E,\pi)$ be the number of elements in the edge set $E$
 such that both endpoints are in the same block of $\pi$. Then we have:
\begin{equation} \label{W2}
  W(\sigma) = \sum_{\pi \in \mathcal{P}(V)}   q^{i(E,\pi)}   \mu(\pi,\hat 1).
\end{equation}
\end{lem}

\begin{proof}
There is a natural bijection between the interval $[\sigma,\hat 1]$ in $\mathcal{P}(2n)$ and the set $\mathcal{P}(V)$,
in such a way that $\cro(\sigma,\pi) = i(E,\pi)$. Hence Equation~\eqref{W2} is just a rewriting of \eqref{W1} in terms
of the graph $G(\sigma)$.
\end{proof}

Now we can use Proposition~\ref{proptutte} from the next section. It allows to recognize $(q-1)^{-n+1}W(\sigma)$
as an evaluation of the Tutte polynomial $T_{G(\sigma)}$, except that it is 0 when the graph is not connected.

Gathering Equations~\eqref{kW}, \eqref{W2}, and Proposition~\ref{proptutte} from the next section,
we have proved:

\begin{thm} \label{cumultutte}
For any $n\geq 1$,
\[
  \tilde k_{2n}(q) = \sum_{\sigma \in \mathcal{M}^c(2n)}  T_{G(\sigma)} (1,q).
\]
In particular $\tilde k_{2n}(q)$ is a polynomial in $q$ with nonnegative coefficients.
\end{thm}

\section{The Tutte polynomial of a connected graph}

For any graph $G=(V,E)$, let $T_G(x,y)$ denote its Tutte polynomial, we give here a short definition
and refer to \cite[Chapter~9]{aigner} for details. This graph invariant can be computed recursively
via edge deletion and edge contraction. Let $e\in E$, let $G\backslash e = (V,E\backslash e)$ and
$G/e = ( V/e , E\backslash e)$ where $V/e$ is the quotient set where both endpoints of the edge $e$
are identified. Then the recursion is:
\begin{equation} \label{recurtutte}
 T_G(x,y) =
 \begin{cases}
  xT_{G/e}(x,y) &  \text{if $e$ is a bridge,} \\
  yT_{G\backslash e}(x,y) &  \text{if $e$ is a loop,} \\
  T_{G/e}(x,y)+T_{G\backslash e}(x,y) &  \text{otherwise.}
 \end{cases}
\end{equation}
The initial case is that $T_G(x,y)=1$ if the graph $G$ has no edge.
Here, a {\it bridge} is an edge $e$ such that $G\backslash e$ has one more connected component than $G$, 
and a {\it loop} is an edge whose both endpoints are identical.

\begin{prop} \label{proptutte}
 Let $G=(V,E)$ be a graph (possibly with multiple edges and loops). Let $n=\#V$.
 With $i(E,\pi)$ defined as in Lemma~\ref{Wgraph}, we have:
\begin{equation} \label{tutte}
  \frac{1}{(q-1)^{n-1}}  \sum_{\pi\in\mathcal{P}(V)} q^{i(E,\pi)}  \mu(\pi,\hat 1)  =
\begin{cases}
T_G(1,q) & \hbox{ if $G$ is connected,} \\
0 & \hbox{otherwise.}
\end{cases}
\end{equation}
\end{prop}

\begin{proof}
Denote by $U_G$ the left-hand side in \eqref{tutte} and let $e$ be an edge of $G$.
Suppose $e\in E$ is a loop, it is then clear that $i(E\backslash e,\pi)=i(E,\pi)-1$, so $U_G = qU_{G\backslash e}$.
Then suppose $e$ is not a loop, and let $x$ and $y$ be its endpoints. We have:
\[
  U_G - U_{G\backslash e} =
  \frac{1}{(q-1)^{n-1}}  \sum_{\pi\in\mathcal{P}(V)} \Big(q^{i(E,\pi)} - q^{i(E\backslash e,\pi)} \Big) \mu(\pi,\hat 1).
\]
In this sum, all terms where $x$ and $y$ are in different blocks
of $\pi$ vanish. So we can keep only $\pi$ such that $x$ and $y$ are in the same block, and these can be
identified with elements of $\mathcal{P}(V/e)$ and satisfy $i(E\backslash e,\pi)=i(E,\pi)-1$. We obtain:
\[
  U_G - U_{G\backslash e} = \frac{1}{(q-1)^{n-2}} \sum_{\pi\in\mathcal{P}(V/e)}  q^{i(E\backslash e,\pi)} \mu(\pi,\hat 1) = U_{G/e}.
\]
This is a recurrence relation which determines $U_G$, and it remains to describe the initial case.
So, suppose the graph $G$ has $n$ vertices and no edge, i.e. $G=(V,\emptyset)$. We have
$i(\emptyset,\pi)=0$. By the definition of the Möbius function, we have:
\[
  \sum_{\pi\in\mathcal{P}(V)}  \mu(\pi,\hat 1) = \delta_{n1},
\]
hence $U_G=\delta_{n1}$ as well in this case.

We have thus a recurrence relation for $U_G$, and it remains to show that the right-hand side of \eqref{tutte}
satisfies the same relation. This is true because when $x=1$, and when we consider a variant of the Tutte
polynomial which is 0 for a non-connected graph, then the first case of \eqref{recurtutte} becomes a particular
case of the third case.
\end{proof}


\begin{rem}
 The proposition of this section can also be derived from results of Burman and Shapiro \cite{burman}, at least in
 the case where $G$ is connected. More precisely, in the light of \cite[Theorem~9]{burman} we can recognize
 the sum in the left-hand side of \eqref{tutte} as the {\it external activity polynomial} $C_G(w)$, where all edge
 variables are specialized to $q-1$. It is known to be related with $T_G(1,q)$, see for example
 \cite[Section 2.5]{sokal}.
\end{rem}

\section{\texorpdfstring{The case $q=0$, Lassalle's sequence and heaps}{The case q=0, Lassalle's sequence and heaps}}
\label{sec:heaps}

In the case $q=0$, the substitution $z \to iz$ recasts Equation~\eqref{relmk} as
\begin{equation} \label{relmk2}
   - \log\bigg(  \sum_{n\geq 0} (-1)^n C_n \frac{z^{2n}}{(2n)!}   \bigg) = \sum_{n\geq 1} \tilde k_{2n}(0) \frac{z^{2n}}{(2n)!},
\end{equation}
where $C_n = \frac{1}{n+1}\tbinom {2n}n $ is the $n$th Catalan number, known to be the cardinal of
$\mathcal{N}(2n)$, see \cite{stanley}. The integer sequence $\{\tilde k_{2n}(0)\}_{n\geq1}=(1,1,5,56,\dots)$
was previously defined by Lassalle \cite{lassalle} via an equation equivalent to \eqref{relmk2}, 
and Theorem 1 from \cite{lassalle} states that the integers $\tilde k_{2n}(0)$ are positive and increasing
(stronger results are also true, see \cite{lassalle,vignat}).

The goal of this section is to give a meaning to \eqref{relmk2} in the context of the theory of heaps \cite{viennot}
\cite[Appendix 3]{cartier}. This will give an alternative proof of Theorem~\ref{cumultutte} for the case $q=0$, based on
a classical result on the evaluation $T_G(1,0)$ of a Tutte polynomial in terms of some orientations of the graph $G$.

\begin{defn}
A graph $G=(V,E)$ is {\it rooted} when it has a distinguished vertex $ r \in V$, called the {\it root}.
An orientation of $G$ is {\it root-connected}, if for any vertex $v\in V$ there exists a directed path
from the root to $v$.
\end{defn}

\begin{prop}[Greene \& Zaslavsky \cite{greene}]  \label{tutte10}
 If $G$ is a rooted and connected graph, $T_G(1,0)$ is the number of its root-connected acyclic orientations.
\end{prop}

The notion of heap was introduced by Viennot \cite{viennot} as a geometric interpretation of elements in
the Cartier-Foata monoid \cite{cartier}, and has various applications in enumeration. We refer to
\cite[Appendix 3]{cartier} for a modern presentation of this subject (and comprehensive bibliography).

Let $M$ be the monoid built on the generators $(x_{ij})_{1\leq i < j}$ subject to the relations
$x_{ij}x_{k\ell} = x_{k\ell} x_{ij} $ if $i<j<k<\ell$ or $i<k<\ell<j$. We call it the Cartier-Foata monoid (but
in other contexts it could be called a partially commutative free monoid or a trace monoid as well).
Following \cite{viennot}, we call an element of $M$ a {\it heap}.

Any heap can be represented as a ``pile'' of segments, as in the left part of Figure~\ref{heapposet}
(this is remindful of \cite{bousquet}).
This pile is described inductively: the generator $x_{ij}$ correspond to a single segment whose extremities
have abscissas $i$ and $j$, and multiplication $m_1m_2$ is obtained by placing the pile of segments
corresponding to $m_2$ above the one corresponding to $m_1$. In terms of segments, the relation 
$x_{ij}x_{k\ell} = x_{k\ell} x_{ij} $ if $i<j<k<\ell$ has a geometric interpretation: segments are allowed to move
vertically as long as they do not intersect (this is the case of $x_{34}$ and $x_{67}$ in Figure~\ref{heapposet}).
Similarly, the other relation $x_{ij}x_{k\ell} = x_{k\ell} x_{ij} $ if $i<k<\ell<j$ can be treated by thinking of
each segment as the projection of an arch as in the central part of Figure~\ref{heapposet}. In this three-dimensional
representation, all the commutation relations are translated in terms of arches that are allowed to move
along the dotted lines as long as they do not intersect.

A heap can also be represented as a poset. Consider two segments $s_1$ and $s_2$ in a pile of segments,
then the relation is defined by saying that $s_1<s_2$ if $s_1$ is always below $s_2$, after any movement of the
arches (along the dotted lines and as long as they do not intersect, as above).
This way, a heap can be identified with a poset where each element is labeled by a generator of $M$, and 
two elements whose labels do not commute are comparable.
See the right part of Figure~\ref{heapposet} for an example and \cite[Appendice 3]{cartier} for details.

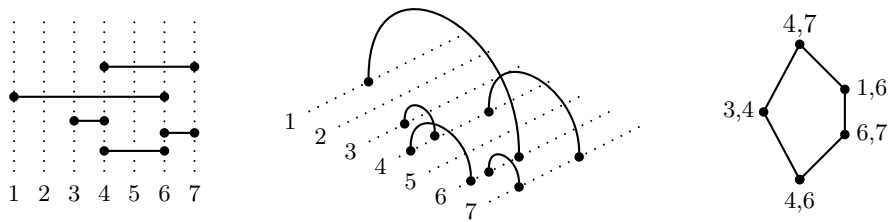
\begin{figure}[h!tp] \psset{unit=4mm}
 \begin{pspicture}(1,-1)(7,5)
    \rput(1,-0.7){\small 1}
    \rput(2,-0.7){\small 2}
    \rput(3,-0.7){\small 3}
    \rput(4,-0.7){\small 4}
    \rput(5,-0.7){\small 5}
    \rput(6,-0.7){\small 6}
    \rput(7,-0.7){\small 7}
    \psline[linestyle=dotted](1,0)(1,5)\psline[linestyle=dotted](2,0)(2,5)\psline[linestyle=dotted](3,0)(3,5)
    \psline[linestyle=dotted](4,0)(4,5)\psline[linestyle=dotted](5,0)(5,5)\psline[linestyle=dotted](6,0)(6,5)
    \psline[linestyle=dotted](7,0)(7,5)
    \psline{*-*}(4,0.7)(6,0.7) \psline{*-*}(6,1.3)(7,1.3) \psline{*-*}(3,1.7)(4,1.7)
    \psline{*-*}(1,2.5)(6,2.5) \psline{*-*}(4,3.5)(7,3.5)
 \end{pspicture}
\hspace{1cm}
\begin{pspicture}(-0.7,-3)(11,5)
\psline[linestyle=dotted](0,0)(5,2.5)
\psline[linestyle=dotted](1,-0.5)(6,2)
\psline[linestyle=dotted](2,-1)(7,1.5)
\psline[linestyle=dotted](3,-1.5)(8,1)
\psline[linestyle=dotted](4,-2)(9,0.5)
\psline[linestyle=dotted](5,-2.5)(10,0)
\psline[linestyle=dotted](6,-3)(11,-0.5)
\rput(-0.6,-0.3){\small 1}
\rput(0.4,-0.8){\small 2}
\rput(1.4,-1.3){\small 3}
\rput(2.4,-1.8){\small 4}
\rput(3.4,-2.3){\small 5}
\rput(4.4,-2.8){\small 6}
\rput(5.4,-3.3){\small 7}
\psbezier{*-*}(3.4,-1.3)(3.4,0.4)(5.4,-0.6)(5.4,-2.3)
\psbezier{*-*}(6,-2)(6,-1)(7,-1.5)(7,-2.5)
\psbezier{*-*}(3.2,-0.4)(3.2,0.6)(4.2,0.2)(4.2,-0.8)
\psbezier{*-*}(2,1)(2,5.5)(7,3)(7,-1.5)
\psbezier{*-*}(6,0)(6,2.5)(9,1)(9,-1.5)
\end{pspicture}
\hspace{1cm} \psset{unit=6mm}
 \begin{pspicture}(-1.5,-0.5)(2,3)
  \psline[showpoints=true]{}(0,0)(1,1)(1,2)(0,3)(-0.8,1.5)(0,0)
  \rput(0,-0.5){\small 4,6}\rput(1.6,1){\small 6,7}\rput(1.6,2){\small 1,6}\rput(-1.35,1.51){\small 3,4}\rput(0,3.4){4,7}
 \end{pspicture}
\caption{The heap $m=x_{46}x_{67}x_{34}x_{16}x_{47}$ as a pile of segments and the Hasse diagram of the associated poset.
\label{heapposet}}
\end{figure}

\begin{defn}
 For any heap $m\in M$, let $|m|$ denote its length as a product of generators.
 Moreover, $m\in M$ is called a {\it trivial heap} if it is a product of pairwise commuting generators.
 Let $M^\circ\subset M $  denote the set of trivial heaps.
\end{defn}

Let $\mathbb{Z}[[M]]$ denote the ring of formal power series in $M$, i.e. all formal sums
$\sum_{m\in M} \alpha_m m$ with multiplication induced by the one of $M$.
A fundamental result of Cartier and Foata \cite{cartier} is the identity in $\mathbb{Z}[[M]]$ as follows:
\begin{equation} \label{cartierfoata}
\bigg( \sum_{m \in M^\circ }   (-1)^{|m|}   m \bigg)^{-1} = \sum_{m\in M} m.
\end{equation}
Note that $M^\circ$ contains the neutral element of $M$ so that the sum in the left-hand side is invertible,
being a formal power series with constant term equal to 1.

\begin{defn} \label{defpyr}
 An element $m\in M$ is called a {\it pyramid} if the associated poset has a unique maximal element.
 Let $P\subset M$ denote the subset of pyramids.
\end{defn}

A fundamental result of the theory of heaps links the generating function of pyramids with the one of all
heaps \cite{cartier,viennot}. It essentially relies on the exponential formula for labeled combinatorial
objects, and reads:
\begin{equation} \label{logpyr1}
  \log \bigg(    \sum_{m\in M} m    \bigg) =_{\text{comm}} \sum_{p \in P} \frac{1}{|p|} p,
\end{equation}
where the sign $=_{\text{comm}}$ means that the equality holds in any commutative
quotient of $\mathbb{Z}[[M]]$. Combining \eqref{cartierfoata} and \eqref{logpyr1}, we obtain:
\begin{equation} \label{logpyr2} 
   -  \log \bigg(    \sum_{m\in M^\circ} (-1)^{|m|} m    \bigg) =_{\text{comm}} \sum_{p \in P} \frac{1}{|p|} p.
\end{equation}
Now, let us examine how to apply this general equality to the present case.

The following lemma is a direct consequence of the definitions, and permits
to identify trivial heaps with noncrossing matchings.

\begin{lem}  \label{Phi}
The map 
\begin{equation} \label{defphi}
   \Phi :   x_{i_1j_1}  \cdots x_{i_nj_n} \mapsto \{\{i_1,j_1\},\dots,\{i_n,j_n\}\}
\end{equation}
defines a bijection between the set of trivial heaps $M^\circ$ and the disjoint union of $\mathcal{N}(V)$
where $V$ runs through the finite subsets (of even cardinal) of $\mathbb{N}_{>0}$.
\end{lem}

For a general heap $m\in M$, we can still define $\Phi(m)$ via \eqref{defphi} but it may not be a matching,
for example $\Phi(x_{1,2}x_{2,3}) = \{\{1,2\},\{2,3\}\}$. Let us first consider the case of $m\in M$ such that
$\Phi(m)$ is really a matching.

\begin{lem} \label{ac_or}
 Let $\sigma\in\mathcal{M}(V)$ for some $V\subset \mathbb{N}_{>0}$. Then the heaps $m\in M$ such that
 $\Phi(m)=\sigma$ are in bijection with acyclic orientations of $G(\sigma)$.
 Thus, such a heap $m\in M$ can be identified with a pair $(\sigma,r)$ where $r$ is an acyclic orientation
 of the graph $G(\sigma)$.
\end{lem}

\begin{proof}
An acyclic orientation $r$ on $G(\sigma)$ defines a partial order on $\sigma$ by saying that two arches $x$ and $y$
satisfy $x<y$ if there is a directed path from $y$ to $x$. In this partial order, two crossing arches are always comparable
since they are adjacent in $G(\sigma)$. We recover the description of heaps in terms of posets, as described above,
so each pair $(\sigma,r)$ corresponds to a heap $m\in M$ with $\Phi(m)=\sigma$.
\end{proof}

To treat the case of $m\in M$ such that $\Phi(m)$ is not a matching, such as $x_{12}x_{23}$,
we are led to introduce a set of commuting variables $(a_i)_{ i \geq 1}$ such that $a_i^2=0$, and consider the specialization
$x_{ij}\mapsto a_ia_j$ which defines a morphism of algebras $\omega : \mathbb{Z}[[M]] \to \mathbb{Z}[[a_1,a_2,\dots]] $.
This way, for any $m\in M$ we have either $\omega(m)=0$, or $\Phi(m) \in \mathcal{M}(V)$
for some $V\subset \mathbb{N}_{>0}$.

Let $m\in M$ such that $\omega(m)\neq0$. As seen in Lemma~\ref{ac_or}, it can be identified with the
pair $(\sigma,r)$ where $\sigma=\Phi(m)$, and $r$ is an acyclic orientation of $G(\sigma)$.
Then the condition defining pyramids is easily translated in terms of $(\sigma,r)$,
indeed we have $m\in P$ if and only if the acyclic orientation $r$ has a unique source
(where a {\it source} is a vertex having no ingoing arrows).

Under the specialization $\omega$, the generating function of trivial heaps is:
\begin{equation} \label{omega1}
\omega\bigg(   \sum_{m \in M^\circ } (-1)^{|m|} m  \bigg) = \sum_{n\geq 0}  (-1)^n C_n e_{2n},
\end{equation}
where $e_{2n}$ is the $2n$th elementary symmetric functions in the $a_i$'s.
Indeed, let $V\subset \mathbb{N}_{>0}$ with $\# V = 2n$, then the coefficient of $\prod_{i\in V} a_i $
in the left-hand side of \eqref{omega1} is $(-1)^n \# \mathcal{N}(V)= (-1)^n C_n$, as can be seen
using Lemma~\ref{Phi}. In particular, it only depends on $n$ so that this generating function can be
expressed in terms of the $e_{2n}$. Moreover, since the variables $a_i$ have vanishing squares their
elementary symmetric functions satisfy
\[
  e_{2n} = \frac{1}{(2n)!}   e_1^{2n},
\]
so that the right-hand side of \eqref{omega1} is actually the exponential generating of the Catalan numbers
(evaluated at $e_1$). It remains to understand the meaning of taking the logarithm of the left-hand side of
\eqref{omega1} using pyramids and Equation~\eqref{logpyr2}.

Note that the relation $=_{\text{comm}}$ becomes a true equality after the specialization $x_{ij}\mapsto a_ia_j$.
So taking the image of \eqref{logpyr2} under $\omega$ and using \eqref{omega1}, this gives
\[ 
  - \log\bigg(  \sum_{n \geq 0} (-1)^n C_n e_{2n}  \bigg) 
  =  \sum_{p\in P} \frac{1}{|p|}  \omega(p).
\]
The argument used to obtain \eqref{omega1} shows as well that the right-hand side of the previous equation
is 
$\sum_{} \frac{x_n}n   e_{2n}$
where $x_n=\#\{ p\in P \; : \; \omega(p)=a_1 \cdots a_{2n} \}$. So we have
\[
  - \log\bigg(  \sum_{n \geq 0} (-1)^n C_n e_{2n}  \bigg) 
  = \sum_{n \geq 0} \frac{x_n}n   e_{2n},
\]
and comparing this with \eqref{relmk2}, we obtain $ \tilde k_{2n} (0) =  \frac {x_n}{n}$.

Clearly, a graph with an acyclic orientation always has a source, and it has a unique source
only when it is root-connected (for an appropriate root, viz. the source). So a pyramid
$p$ such that $\omega(p)\neq0$ can be identified with a pair $(\sigma,r)$ where $r$ is a
root-connected acyclic orientation of $G(\sigma)$. Then using Proposition~\ref{tutte10}, it follows that
\[
  x_n = n  \sum_{\sigma \in \mathcal{M}^c(2n) } T_{G(\sigma)}(1,0).
\]
Here, the factor $n$ in the right-hand side accounts for the $n$ possible choices of the source
in each graph $G(\sigma)$. Eventually, we obtain
\begin{equation} \label{cumultutte0}
  \tilde k _{2n}(0) = \sum_{\sigma \in \mathcal{M}^c(2n) } T_{G(\sigma)} (1,0),
\end{equation}
i.e. we have proved the particular case $q=0$ of Theorem~\ref{cumultutte}.

Let us state again the result in an equivalent form. We can consider that if $\sigma\in\mathcal{M}(2n)$,
the graph $G(\sigma)$ has a canonical root which the arch containing 1. Then, Equation \eqref{cumultutte0}
gives a combinatorial model for the integers $\tilde k_{2n}(0)$:

\begin{thm}
 The integer $\tilde k_{2n}(0)$ counts pairs $(\sigma,r)$ where $\sigma\in\mathcal{M}^c(2n)$, and $r$
 is an acyclic orientation of $G(\sigma)$ whose unique source 
 is the arch of $\sigma$ containing 1.
\end{thm}

From this, it is possible to give a combinatorial proof that the integers $\tilde k_{2n}(0)$ are increasing,
as suggested by Lassalle \cite{lassalle} who gave an algebraic proof. Indeed, we can check that pairs
$(\sigma,r)$ where $\{1,3\}$ is an arch of $\sigma$ are in bijection with the same objects but of size one
less, hence $\tilde k_{2n}(0) \leq \tilde k_{2n+2}(0)$.

Before ending this section, note that the left-hand side of \eqref{relmk2} is  $-\log( \frac 1z J_1(2z))$
where $J_1$ is the Bessel function of order 1. There are quite a few other cases where the combinatorics
of Bessel functions is related with the theory of heaps, see the articles of Fédou \cite{fedou1,fedou2},
Bousquet-Mélou and Viennot \cite{bousquet}.

\section{\texorpdfstring{The case $q=2$, the exponential formula}{The case q=2, the exponential formula}}
\label{sec:q=2}

The specialization at $(1,2)$ of a Tutte polynomial has combinatorial significance in terms of 
connected spanning subgraphs (see \cite[Chapter 9]{aigner}), so it is natural to consider
the case $q=2$ of Theorem~\ref{cumultutte}. This case is particular because the factor $(q-1)^{n-1}$ disappears, so
that $\tilde k_{2n}(2) = k_{2n}(2)$. We can then interpret the logarithm in the sense of combinatorial species, by
showing that $\tilde k_{2n}(2)$ counts some {\it primitive} objects and $m_{2n}(2)$ counts {\it assemblies} of those,
just like permutations that are formed by assembling cycles (this is the exponential formula for labeled combinatorial
objects, see \cite[Chapter 3]{aigner}). What we obtain is another more direct proof of Theorem~\ref{cumultutte}, based
on an interpretation of $T_G(1,2)$ as follows.

\begin{prop}[Gioan \cite{gioan}]  \label{propgioan}
 If $G$ is a rooted and connected graph, $T_G(1,2)$ is the number of its root-connected orientations.
\end{prop}

This differs from the more traditional interpretation of $T_G(1,2)$ in terms of connected
spanning subgraphs mentioned above, but it is what naturally appears in this context.

\begin{defn}
 Let $\mathcal{M}^+(2n)$ be the set of pairs $(\sigma,r)$ where $\sigma\in\mathcal{M}(2n)$ and $r$ is an
 orientation of the graph $G(\sigma)$. Such a pair is called an {\it augmented matching}, and is depicted
 with the convention that the arch $\{i,j\}$ lies above the arch $\{k,\ell\}$ if there is an oriented edge
 $\{i,j\} \rightarrow \{k,\ell\}$, and behind it if there is an oriented edge $\{k,\ell \} \rightarrow \{i,j\}$
. 
\end{defn}

See Figure~\ref{aug} for example. Clearly, $\#\mathcal{M}^+(2n) = m_{2n}(2)$. Indeed, each graph
$G(\sigma)=(V,E)$ has $2^{\# E}$ orientations, and $\# E= \cro(\sigma)$, so this follows from \eqref{mucro}.

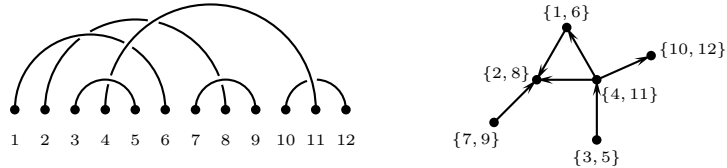
\begin{figure}[h!tp] \psset{unit=4mm}
\begin{pspicture}(1,-1)(12,3.5)
  \psdots(1,0)(2,0)(3,0)(4,0)(5,0)(6,0)(7,0)(8,0)(9,0)(10,0)(11,0)(12,0)
  \rput(1,-1){\tiny 1}\rput(2,-1){\tiny 2}\rput(3,-1){\tiny 3}\rput(4,-1){\tiny 4}\rput(5,-1){\tiny 5}
  \rput(6,-1){\tiny 6}\rput(7,-1){\tiny 7}\rput(8,-1){\tiny 8}\rput(9,-1){\tiny 9}\rput(10,-1){\tiny 10}
  \rput(11,-1){\tiny 11}\rput(12,-1){\tiny 12}
  \psarc(8,0){1}{0}{180}\psarc(11,0){1}{0}{88}\psarc(11,0){1}{107}{180}
  \psarc(4,0){1}{0}{180}\psarc(7.5,0){3.5}{0}{161}\psarc(7.5,0){3.5}{166}{180}
  \psarc(3.5,0){2.5}{0}{56}\psarc(3.5,0){2.5}{63}{180}
  \psarc(5,0){3}{0}{16}\psarc(5,0){3}{22}{75}\psarc(5,0){3}{82}{120}\psarc(5,0){3}{128}{180}
\end{pspicture}
\hspace{1.5cm} \psset{unit=8mm}
\begin{pspicture}(4,3)
  \psdots(1,1)(2,1)(1.5,1.87)(0.29,0.29)(2,0)(2.9,1.4)
  \psline[arrowlength=2]{<-}(1,1)(1.5,1.87)\psline[arrowlength=2]{->}(2,1)(1.5,1.87)
  \psline[arrowlength=2]{<-}(1,1)(2,1)\psline[arrowlength=2]{<-}(1,1)(0.29,0.29)
  \psline[arrowlength=2]{<-}(2,1)(2,0)\psline[arrowlength=2]{->}(2,1)(2.9,1.4)
  \rput(1.5,2.1){\tiny \{$1,6$\}}\rput(0.5,1.1){\tiny \{$2,8$\}}
  \rput(0,0){\tiny \{$7,9$\}}\rput(2,-0.3){\tiny \{$3,5$\}}
  \rput(2.55,0.75){\tiny \{$4,11$\}}\rput(3.6,1.5){\tiny \{$10,12$\}}
\end{pspicture}
\caption{An augmented matching $(\sigma,r)$ and the corresponding orientation of $G(\sigma)$. \label{aug}}
\end{figure}

Notice that if there is no directed cycle in the oriented graph $(G(\sigma),r)$, the augmented
matching $(\sigma,r)$ can be identified with a heap $m\in M$ as defined in the previous section.
The one in Figure~\ref{aug} would be $x_{3,5}x_{4,11}x_{10,12}x_{1,6}x_{7,9}x_{2,8}$.
Actually, the application of the exponential formula in the present section is quite reminiscent of
the link between heaps and pyramids as seen in the previous section.

\begin{defn}
 Recall that each graph $G(\sigma)$ is rooted with the convention that the root is the arch containing 1.
 Let $\mathcal{I}(2n) \subset \mathcal{M}^+(2n)$ be the set of augmented matchings $(\sigma,r)$ such that
 $\sigma$ is connected and $r$ is a root-connected orientation of $G(\sigma)$. The elements of $\mathcal{I}(2n)$
 are called {\it primitive} augmented matchings. For any $V\subset \mathbb{N}_{>0}$ with $\#V=2n$, we also
 define the set $\mathcal{I}(V)$, with the same combinatorial description as $\mathcal{I}(2n)$ except that
 matchings are based on the set $V$ instead of $\{1,\dots,2n\}$.
\end{defn}

Using Proposition~\ref{propgioan}, we have 
\[
  \# \mathcal{I}(2n) = \sum_{\sigma\in\mathcal{M}^c(2n)} T_{G(\sigma)}(1,2),
\]
so that the particular case $q=2$ of Theorem~\ref{cumultutte} is the equality $\# \mathcal{I}(2n) = k_{2n}(2)$.
To prove this from \eqref{relmk} and using the exponential formula, we have to see how an augmented
matching can be decomposed into an assembly of primitive ones, as stated in Proposition~\ref{propdecomp}
below. This decomposition thus proves the case $q=2$ of Theorem~\ref{cumultutte}.
Note also that the bijection given below is equivalent to the first identity in \eqref{inversion}.

\begin{prop} \label{propdecomp}
 There is a bijection
\[
   \mathcal{M}^+(2n) \longrightarrow \biguplus_{\pi\in\mathcal{P}(n)} \; \prod_{ V\in \pi } \mathcal{I}(V).
\]
\end{prop}

\begin{proof}
Let $(\sigma,r) \in \mathcal{M}^+(2n)$, the bijection is defined as follows. Consider the vertices of $G(\sigma)$ which are 
accessible from the root. This set of vertices defines a matching on a subset $V_1\subset \{1,\dots,2n\}$. For example,
in the case in Figure~\ref{aug}, the root is $\{1,6\}$ and the only other accessible vertex is $\{2,8\}$, so $V_1=\{1,2,6,8\}$.
Together with the restriction of the orientation $r$ on this subset of vertices, this defines an augmented matching
$(\sigma_1,r_1)\in\mathcal{M}^+(V_1)$ which by construction is primitive. By repeating this operation on the set
$\{1,\dots,2n\}\backslash V_1$, we find $V_2\subset \{1,\dots,2n\}\backslash V_1$ and $(\sigma_2,r_2)\in\mathcal{I}(V_2)$,
and so on. See Figure~\ref{decomp} for the result, in the case of the augmented matching in Figure~\ref{aug}.

The inverse bijection is easily described. If $(\sigma_i,r_i)\in\mathcal{I}(V_i)$ for any $1\leq i\leq k$ where $\pi=\{V_1,\dots,V_k\}$,
let $\sigma=\sigma_1 \cup \dots \cup \sigma_k$, and the orientation $r$ of $G(\sigma)$ is as follows. Let $e$ be an edge of
$G(\sigma)$ and $x_1$, $x_2$ be its endpoints, with $x_1\in\sigma_{j_1}$ and $x_2\in\sigma_{j_2}$. If $j_1=j_2$, the edge
$e$ is oriented in accordance with the orientation $r_{j_1}=r_{j_2}$. Otherwise, say $j_1<j_2$, then the edge $e$ is oriented
in the direction $x_1 \leftarrow x_2$.
\end{proof}

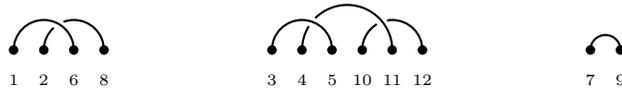
\begin{figure}[h!tp] \psset{unit=4mm}
\begin{pspicture}(1,-1)(4,1.5)
  \psdots(1,0)(2,0)(3,0)(4,0)
  \psarc(2,0){1}{0}{180}\psarc(3,0){1}{0}{110}\psarc(3,0){1}{132}{180}
  \rput(1,-1){\tiny 1}\rput(2,-1){\tiny 2}\rput(3,-1){\tiny 6}\rput(4,-1){\tiny 8}
\end{pspicture}
\hspace{2cm}
\begin{pspicture}(1,-1)(6,1.5)
  \psdots(1,0)(2,0)(3,0)(4,0)(5,0)(6,0)
  \rput(1,-1){\tiny 3}\rput(2,-1){\tiny 4}\rput(3,-1){\tiny 5}
  \rput(4,-1){\tiny 10}\rput(5,-1){\tiny 11}\rput(6,-1){\tiny 12}
  \psarc(2,0){1}{0}{180}
  \psarc(3.5,0){1.5}{0}{135}\psarc(3.5,0){1.5}{148}{180}
  \psarc(5,0){1}{0}{100}\psarc(5,0){1}{121}{180}
\end{pspicture}
\hspace{2cm}
\begin{pspicture}(1,-1)(2,1.5)
  \psdots(1,0)(2,0)
  \psarc(1.5,0){0.5}{0}{180}
  \rput(1,-1){\tiny 7}\rput(2,-1){\tiny 9}
\end{pspicture}
\caption{Decomposition of an augmented matching into primitive ones. \label{decomp}}
\end{figure}

\section{Cumulants of the free Poisson law}

The free Poisson law appears in free probability and random matrices theory, and can be characterized by the fact that 
all free cumulants are equal to some $\lambda>0$, see \cite{nica}. It follows from \eqref{inversionfree}  that its 
moments $m_n(\lambda)$ count noncrossing partitions, and consequently the coefficients are given by the Narayana 
numbers (see \cite{stanley}):
\[
  m_n(\lambda)= \sum_{\pi\in\mathcal{NC}(n)} \lambda^{\# \pi} = \sum_{k=1}^n \frac {\lambda^k}{n} \binom{n}{k}\binom{n}{k-1}.
\]
The corresponding cumulants are as before defined by 
\begin{equation}  \label{defklam}
  \sum_{n\geq1} k_n(\lambda) \frac{z^n}{n!} =\log \Bigg( \sum_{n\geq0} m_n(\lambda) \frac{z^n}{n!} \Bigg).
\end{equation}
For any set partition $\pi\in\mathcal{P}(V)$ for some $V\subset \mathbb{N}$, we can define a crossing graph $G(\pi)$, 
whose vertices are the blocks of $\pi$, and there is an edge between $b,c\in\pi$ if $\{b,c\}$ is not a noncrossing 
partition. Note that $\pi$ is connected if and only if the graph $G(\pi)$ is connected.
The two different proofs for the semicircular cumulants show as well the following:

\begin{thm} \label{cumulpoisson}
For any $n\geq1$, we have:
\[
  k_n(\lambda) = - \sum_{\pi\in\mathcal{P}^c(n)}  (-\lambda)^{\# \pi } T_{G(\pi)}(1,0).
\]
\end{thm}

Let us sketch the proofs. If $\pi\in\mathcal{P}(n)$, similar to Lemma~\ref{lemmpi} we have:
\[
  m_{\pi}(\lambda) = \sum_{\substack{ \rho \in \mathcal{P}(n) \\ \rho \unlhd \pi } } \lambda^{\# \rho}
\]
where the relation $\rho \unlhd \pi $ means that $\rho \leq \pi $ and $\rho|_b$ is a noncrossing
partition for each $b\in\pi$. Indeed the map $\rho \mapsto (\rho|_b)_{b\in\pi}$ is a bijection
between $\{ \rho\in\mathcal{P}(n) : \rho \unlhd \pi \}$ and $\prod_{b\in\pi} \mathcal{NC}(b)$.
The same computation as in \eqref{kW} and \eqref{W1} gives
\begin{equation} \label{invpoisson}
    k_n(\lambda) = \sum_{\pi\in\mathcal{P}(n)} m_{\pi}(\lambda) \mu(\pi,\hat 1) 
  = \sum_{ \substack{ \rho,\pi \in \mathcal{P}(n) \\ \rho \unlhd \pi } } \lambda^{ \# \rho } \mu(\pi,\hat 1)
  = \sum_{\rho \in \mathcal{P}(n) }  \lambda^{\# \rho}  W(\rho),
\end{equation}
where 
\[
  W(\rho) = \sum_{ \substack{ \pi \in \mathcal{P}(n) \\ \rho \unlhd \pi } } \mu(\pi,\hat 1).
\]
Denoting $G(\rho)=(V,E)$ the crossing graph of $\rho$, the previous equality is rewritten
$W(\rho) = \sum  \mu(\pi,\hat 1)$ where the sum is over $\pi\in\mathcal{P}(V)$ such that
for any $b\in\pi$, $e\in E$, the block $b$ does not contain both endpoints of the edge $e$.
Then the case $q=0$ of Proposition~\ref{proptutte} shows that
\[ 
   W(\rho) = \begin{cases}
               (-1)^{\# \rho +1 } T_{G(\rho)}(1,0) & \text{if } \rho\in\mathcal{P}^c(n), \\
               0  & \text{otherwise.}
             \end{cases}
\]
Together with \eqref{invpoisson}, this completes the first proof of Theorem~\ref{cumulpoisson}.

As for the second proof, we follow the outline of Section~\ref{sec:heaps}, but with another definition 
for $M$, $M^\circ$, $P$ and $\omega$. Let $M$ be the monoid with generators 
$(x_V)$ where $V$ runs through finite subsets of $\mathbb{N}_{>0}$, and with relations $x_Vx_W=x_Wx_V$ if 
$\{V,W\}$ is a noncrossing partition. We also denote $M^\circ \subset M$ the corresponding set of trivial heaps, 
i.e. products of pairwise commuting generators. The subset $P\subset M$ is characterized by Definition~\ref{defpyr}.
Now, we consider the morphism $\omega$ defined by
\[
  \omega(x_V) = \lambda \prod_{i\in V} a_i.
\]
We have:
\begin{align*}
   \omega \bigg( \sum_{m\in M^\circ}  (-1)^{|m|} m \bigg) 
  &= \sum_{V} \sum_{\pi\in\mathcal{NC}(V)} (-1)^{\#\pi} \prod_{b\in \pi} \omega(x_b) \\
  &= \sum_{V} \sum_{\pi\in\mathcal{NC}(V)} (-\lambda)^{\#\pi} \prod_{i\in V} a_i \\
  &= \sum_{n\geq 0} m_n(-\lambda) e_{n}
  = \sum_{n\geq 0} m_n(-\lambda) \frac{e_{1}^n}{n!}.
\end{align*}
We still understand that $V\subset \mathbb{N}_{>0}$ is finite,
$(a_i)_{i\geq 1}$ are commuting variables with vanishing squares, and $e_n$ is the $n$th
elementary symmetric function in the $a_i$'s.
Equation \eqref{logpyr2} is still valid as such with the new definition of $M^\circ$ and $P$,
and taking the image by $\omega$ gives:
\[
  -\log \bigg( \sum_{n\geq 0} m_n(-\lambda) \frac{e_{1}^n}{n!} \bigg) = \sum_{p\in P} \frac{1}{|p|} \omega(p).
\]
Comparing with \eqref{defklam} and taking the coefficient of $e_1^n$, we get:
\[
  -k_n(-\lambda) =   \sum_{\substack{ p\in P \\ \omega(p) = a_1 \cdots a_n  }} \frac{\lambda^{|p|}}{|p|} .
\]
Let $p\in P$ be such that $\omega(p) = a_1 \cdots a_n$. Following the idea in Lemma~\ref{ac_or},
we can write $p=x_{V_1}\cdots x_{V_k}$ where $V_1,\dots,V_k$ are the blocks of a set partition 
$\pi\in\mathcal{P}(n)$, and $p$ is characterized by $\pi$ together with an acyclic orientation 
of the graph $G(\pi)$ having a unique source. Following the idea at the end of Section~\ref{sec:heaps}, 
we thus complete the second proof of Theorem~\ref{cumulpoisson}.

\section{Final remarks}

It would be interesting to explain why the same combinatorial objects appear both for 
$c_{2n}(q)$ and $k_{2n}(q)$. This suggests that there exists some quantity that interpolates between 
the classical and free cumulants of the $q$-semicircular law, however, building a noncommutative 
probability theory that encompasses the classical and free ones appear to be elusive (see \cite{leeuwen2}
for a precise statement). It means that building such an interpolation would rely not only on the
$q$-semicircular law and its moments, but on its realization as a noncommutative random variable.
This might be feasible using $q$-Fock spaces \cite{bozejko1,bozejko2} but is beyond the scope
of this article.

\section*{Acknowledgment}

This work was initiated during the trimester ``Bialgebras in Free Probability'' at the Erwin Schrödinger 
Institute in Vienna. In particular I thank Franz Lehner, Michael Anshelevich and Natasha Blitvić for 
their conversation.


\medskip

\setlength{\parindent}{0pt}


\begin{thebibliography}{999}

\bibitem{aigner}
 M. Aigner:
 {\it A course in enumeration}.
 Springer, Berlin, 2007.

\bibitem{vignat}
 T. Amdeberhan, V.H. Moll and C. Vignat:
 A new proof of a conjecture by D. Zeilberger about Catalan numbers.
 Preprint, available as arXiv:1202.1203v1.

\bibitem{anshelevitch}
 M. Anshelevich, S. T. Belinschi, M. Bożejko and F. Lehner:
 Free infinite divisibility for $q$-Gaussians.
 Math. Res. Lett. 17 (2010), 905--916.

\bibitem{bousquet}
 M. Bousquet-Mélou and X.G. Viennot:
 Empilements de segments et $q$-énumération de polyominos convexes dirigés.
 J. Combin. Theory Ser. A 60 (1992), 196--224.

\bibitem{belinschi}
 S. Belinschi, M. Bożejko, F. Lehner and R. Speicher:
 The normal distribution is $\boxplus$-infinitely divisible.
 Adv. Math. 226 (2011), 3677--3698.

\bibitem{blitvic}
 N. Blitvić: 
 On the norm of $q$-circular operators. 
 Preprint.

\bibitem{bozejko1}
 M. Bożejko, B. Kümmerer and R. Speicher:
 $q$-Gaussian processes: non-commutative and classical aspects.
 Comm. Math. Phys. 185 (1997), 129--154.

\bibitem{bozejko2}
 M. Bożejko and R. Speicher:
 An example of a generalized brownian motion.
 Comm. Math. Phys. 137 (1991), 519--531.

\bibitem{burman}
 Y. Burman and B. Shapiro:
 Around matrix-tree theorem.
 Math. Res. Lett. 13 (2006), 761--774.

\bibitem{cartier}
 P. Cartier and D. Foata:
 Problèmes combinatoires de commutation et réarrangements.
 Lecture Notes in Math., Vol. 85, Springer-Verlag, Berlin, 1969.
 Electronic reedition with three new appendices, 2006.

\bibitem{fedou1}
 J.-M. Fédou:
 Sur les fonctions de Bessel.
 Discrete Math. 139 (1995), 473--480.

\bibitem{fedou2}
 J.-M. Fédou:
 Combinatorial objects enumerated by $q$-Bessel functions.
 Rep. Math. Phys. 34 (1994), 57--70.

\bibitem{gioan}
 E. Gioan:
 Enumerating degree sequences in digraphs and a cycle-cocycle reversing system.
 European J. Comb. 28 (2007) 1351--1366.

\bibitem{greene}
 C. Greene and T. Zaslavsky:
 On the interpretation of Whitney numbers through arrangements  of hyperplanes, zonotopes, non-Radon partitions, and orientations of graphs.
 Trans. Am. Math. Soc. 280 (1982), 97–126.

\bibitem{hiai}
 F. Hiai and D. Petz:
 {\it The Semicircle Law, Free Random Variables and Entropy}.
 American Mathematical Society, Providence, 2000.

\bibitem{ismail}
 M.E.H. Ismail, D. Stanton and X. G. Viennot:
 The combinatorics of $q$-Hermite polynomials and the Askey-Wilson integral.
 European J. Combin. 8 (1987), 379--392.



\bibitem{lassalle}
 M. Lassalle:
 Two integer sequences related to Catalan numbers.
 J. Combin. Theory Ser. A 119 (2012), 923--935.
 
\bibitem{lehner}
 F. Lehner:
 Free cumulants and enumeration of connected partitions.
 Europ. J. Combin. 23 (2002), 1025--1031.

\bibitem{leeuwen1}
 H. van Leeuwen and H. Maassen:
 A $q$-deformation of the Gauss distribution.
 J. Math. Phys. 36 (1995), 4743--4756.

\bibitem{leeuwen2}
 H. van Leeuwen and H. Maassen:
 An obstruction for $q$-deformation of the convolution product.
 J. Phys. A 29 (1996), 4741--4748.

\bibitem{nica}
 A. Nica and R. Speicher,
 Lectures on the combinatorics of free probability.
 London Math. Soc. Lecture Note Ser., Vol. 335. 
 Cambridge University Press, Cambridge, 2006.

\bibitem{szego}
 G. Szegő:
 Ein Beitrag zur Theorie der Thetafunktionen.
 Sitz. Preuss. Akad. Wiss. Phys. Math. Kl. 19 (1926), 242--252.

\bibitem{touchard}
 J. Touchard:
 Sur un problème de configurations et sur les fractions continues.
 Can. Jour. Math. 4 (1952), 2--25.


\bibitem{sokal}
 A.D. Sokal:
 The multivariate Tutte polynomial (alias Potts model) for graphs and matroids.
 In Surveys in combinatorics 2005, London Math. Soc. Lecture Note Ser., Vol. 327,
 pp. 173--226. Cambridge University Press, Cambridge, 2005.

\bibitem{stanley}
 R.P. Stanley:
 {\it Enumerative Combinatorics, Volume 2}.
 Cambridge University Press, Cambridge, 1999.

\bibitem{viennot}
 X.G. Viennot:
 Heaps of pieces I. Basic definitions and combinatorial lemmas.
 Lecture Notes in Math., Vol. 1234, pp. 321--350. Springer, Berlin, 1986.


\end{thebibliography}
\end{document}